\documentclass[11 pt,a4paper]{article}
\usepackage{amsmath,amscd}
\usepackage{amstext}
\usepackage{amssymb,epsfig}
\usepackage{amsthm}
\usepackage{amsfonts}
\usepackage{graphicx}
\usepackage[all]{xy}
\usepackage{tikz}

\renewcommand{\footnote}{\endnote}
\newtheorem{theorem}{Theorem}[section]

\newtheorem{lemma}[theorem]{Lemma}
\newtheorem{proposition}[theorem]{Proposition}
\newtheorem{corollary}[theorem]{Corollary}

\theoremstyle{definition}
\newtheorem{definition}[theorem]{Definition}
\newtheorem{remark}[theorem]{Remark}

\newtheorem{example}[theorem]{Example}

\usepackage[utf8]{inputenc}      
\begin{document}
\date{}
\title{Classification of rotations on the torus $\mathbb{T}^2$}
\author{Nicolas Bedaride
\footnote{Fédération de recherches des unités de mathématiques de Marseille,
Laboratoire d'Analyse topologie, et probabilités  UMR 7353, Case cours A, Avenue Escadrille Normandie Niemen 13397 Marseille cedex 20,
France. nicolas.bedaride@univ-amu.fr}}

\maketitle
\bibliographystyle{alpha}
\begin{abstract}
We consider rotations on the torus $\mathbb{T}^2$, and we classify
them with respect to the  complexity functions. In dimension one, a minimal rotation  can be coded by a sturmian word. A sturmian word has complexity $n+1$ by the Morse-Hedlund theorem. Here we
make a generalization in dimension two.
\end{abstract}
\section{Introduction}
Sturmian words are infinite words over a two-letter alphabet that have exactly $n+1$ factors of length $n$ for each integer $n$. The number of factors, of a given length, of an infinite word is called the complexity function. These words have been introduced by Morse-Hedlund, \cite{He.Mo}. 
They admit several equivalent definitions.
 One possibility is to consider the rotation of angle $\alpha$ on the torus $\mathbb{T}^1$.  Consider a two-letter alphabet corresponding to the intervals $(0;1-\alpha)$ and $(1-\alpha;1)$. Then, the orbit of any point under the rotation is coded by an infinite word. This word is a sturmian word if and only if $\alpha$ is an irrational number. 
 If $\alpha$ is rational, then the rotation is periodic and the word is periodic.
 
 A sturmian word can also be defined by the billiard map:
A billiard ball, i.e.\ a point mass, moves inside a polyhedron $P$
with unit speed along a straight line until it reaches the
boundary $\partial{P}$, then it instantaneously changes direction
according to the mirror law, and continues along the new line.
Label the faces of $P$ by symbols from a finite alphabet
$\mathcal{A}$ whose cardinality equals the number of faces of $P$.
In the case of the square, we can code the parallel faces by the same letter, and the orbit of a point is coded by an infinite word on two letters. Since the work of Coven-Hedlund, \cite{Co.He},  we know that this word is a sturmian word (under suitable hypothesis on the direction).  
Thus we have three equivalent definitions for a sturmian word:  
An infinite word with $n+1$ factors of length $n$; an infinite word which codes the orbit of a point under an irrational rotation on the torus $\mathbb{T}^1$; an infinite word which codes the orbit of a point under the billiard orbit inside the square.  

 Several attempts have been made to extend Sturmian words to words over alphabets  of more than two letters. An approach has been initiated by Rauzy \cite{Ra.83}, and 
 developped by Arnoux, Rauzy \cite{Ar.Rau}. Here we present an other approach. We consider a rotation on the torus of dimension two, with a natural partition of the torus, see \cite{Ra.82}. The orbit of a point is coded by an infinite word. As in the one-dimensional case, this map can be seen as a billiard orbit inside the cube coded with three letters. Thus the infinite word can be viewed as the coding of the billiard orbit of a point inside the cube. 
The computation of the complexity has been made when the direction satisfies some algebraic
conditions. Under these assumptions the complexity equals
$n^2+n+1$. The first proof was given in \cite{Ar.Ma.Sh.Ta,
Ar.Ma.Sh.Ta1}, and a general proof in dimension $s\geq 3$ appears
in \cite{Ba}. Moreover, we give a new proof of the three-dimensional
result in \cite{moi1}, and we remark that the proof of
\cite{Ar.Ma.Sh.Ta, Ar.Ma.Sh.Ta1} is false: there exists a minimal
direction with a complexity less than $n^2+n+1$. 

In this paper we give a complete characterization of the complexity of two-dimensional rotations, and
obtain the complexity of a billiard word in the non-totally irrational cases. Moreover our study allows us to describe the geometry of a non minimal rotation orbit. In most of the cases it reduces to a one-dimensional translation. The proof could be generalized to higher dimension, the scheme of the proof is the same using $d$-dimensional translations instead of one dimensional translations.

\subsection{Outline of the paper}

In this paper we consider a rotation on the torus $\mathbb{T}^2$.  This rotation is related to a billiard orbit inside the cube. We consider a point, and its orbit under the rotation. It gives an infinite word, its complexity is function of the angle of the rotation. We classify these complexities under the hypothesis fulfilled by the angle of rotation. Since the angle of rotation is linked to the direction of the billiard orbit, we express the hypothesis in term of the direction.

In Section \ref{background} we recall some definitions of combinatorics, of the billiard map, and some results about the complexity of billiard words. 
In Section \ref{enon} we give the statement of the theorem.
In Section \ref{proof} we prove our result. We split the proof in different propositions which represent the different cases of the theorem. Most of time the proof consists in a reduction to the one dimensional-case.

\section{Background}\label{background}
\subsection{Combinatorics}
\begin{definition}
Let $\mathcal{A}$ be a finite set called the \emph{alphabet}. By a
language $L$ over $\mathcal{A}$ we always mean a factorial
extendable language: a language is a collection of sets
$(L_n)_{n\geq 0}$ where the only element of $L_0$ is the empty
word, and each $L_n$ consists of words of the form $a_1a_2\dots
a_n$ where $a_i\in\mathcal{A}$ and such that for each $v\in L_n$
there exist $a,b\in\mathcal{A}$ with $av,vb\in L_{n+1}$, and for
all $v\in L_{n+1}$ if $v=au=u'b$ with $a,b\in\mathcal{A}$ then
$u,u'\in L_n$.
\end{definition}

\begin{remark}
This definition is closely related to the lamination language defined in \cite{Cou.Hil.Lus.07}.
\end{remark}
\begin{definition}
Let $\mathcal{L}$ be an extendable, factorial language.
The complexity function of the language $\mathcal{L}$ is
defined by 
$$p:\mathbb{N}\rightarrow\mathbb{N}$$ 
$$p(n)=card(L_n).$$
\end{definition}
\begin{definition}\label{compu}
An infinite word $v$ over the alphabet $\mathcal{A}$ is a sequence $(v_n)_{n\in\mathbb{N}}$ such that $v_n\in\mathcal{A}$ for every integer $n$.  A subword $w$ of $v$ of length $n$ is a finite word such that there exists $n_0\in\mathbb{N}$ and $w=v_{n_0}v_{n_0+1}\dots v_{n_0+n-1}$. The set of subwords of length $n$ is denoted by $\mathcal{L}_n$.
If $v$ is an infinite word defined over a finite alphabet, then the union $L=\displaystyle\bigcup\mathcal{L}_n$ forms a language. The complexity of $u$ is by definition the complexity of $L$.
\end{definition}

\subsection{Billiard map}
We recall some facts from billiard theory. Additional details can be found in \cite{Ta} or \cite{Ma.Tab}. 

\begin{definition}\label{bordcube}
Let $C$ be the cube $[0;1]^3$, we denote by $(e_i)_{0\leq i\leq3}$ the orthonormal basis of $\mathbb{R}^3$, and by $<e_i,e_j>$ the  square generated by the vectors $e_i,e_j$. 
$$<e_i,e_j>=\{\lambda e_i+\mu e_j, 0\leq \lambda\leq 1; 0\leq \mu\leq 1\}.$$

Then $\partial{C}$ is the union of the six following squares:
$$<e_1,e_2>+(0,0,0); <e_1,e_2>+(0,0,1);$$ 
$$<e_1,e_3>+(0,0,0); <e_1,e_3>+(0,1,0 );$$
$$ <e_2,e_3>+(0,0,0); <e_2,e_3>+(1,0,0).$$
It is the boundary of the cube.
\end{definition}

A billiard ball, i.e.\ a point mass, moves inside $C$ with unit speed along a straight
line until it reaches the boundary $\partial{C}$ (see Definition \ref{bordcube}), then
instantaneously changes direction according to the mirror law, and
continues along the new line. More precisely, the billiard map $T$ is defined
on a subset $X$ of $\partial{C}\times \mathbb{RP}^2$ by
the following method (where $\mathbb{RP}^2$ is the projective plane):

First we define the set $X'\subset\partial{C}\times\mathbb{RP}^2$.
A point $(m,\omega)$ belongs to $X'$ if and only if one of the two following conditions holds: 

\begin{enumerate}
\item The line $m+\mathbb{R}[\omega]$ intersects an edge of
$C$, where $[\omega]$ is a vector of $\mathbb{R}^3$ which
represents $\omega$.\\

 \item The line $m+\mathbb{R}[\omega]$ is included inside
 the face of $C$ which contains $m$.
\end{enumerate}
 Then we define $X$ as the set
 $$X=(\partial{C}\times
 \mathbb{RP}^2)\setminus X'.$$

 Now we define the map $T$:
Consider $(m,\omega)\in X$, then we have
$T(m,\omega)=(m',\omega')$ if and only if $mm'$ is colinear to
$[\omega]$, and $[\omega']=s[\omega]$, where $s$ is the linear
reflection over the face which contains $m'$.
$$T:X\rightarrow \partial{C}\times\mathbb{RP}^3$$
$$T:(m,\omega)\mapsto (m',\omega')$$
\begin{remark}
In the sequel we identify $\mathbb{RP}^2$ with the unit vectors
of $\mathbb{R}^3$ ({\it i.e} we identify $\omega$ and
$[\omega]$).
\end{remark}

\subsection{Notations for the billiard map}
Label the faces of $C$ by three symbols from a finite alphabet
$\mathcal{A}$ such that two parallel faces of the cube are coded
by the same symbols. To the orbit of a point in a direction
$\omega$, we associate a word in the alphabet $\mathcal{A}$
defined by the sequence of faces of the billiard trajectory.

\begin{definition}\label{defbill}
The set of points $(m,\omega)$ such that for all integers $n$,
$T^{n}(m,\omega)\in X$ is denoted by $X_\infty$. The infinite word
associated with a point $(m,\omega)$ in $X_\infty$ is denoted by
$v_{m,\omega}$.
\end{definition}

We finish by three definitions used in the proof of Lemma \ref{ordre}.
\begin{definition}\label{numeroedge}
 An edge parallel to the axis $Ox$, resp., $Oy$, $Oz$
is called of type 1, (type 2, type 3, resp.)
\end{definition}

\begin{definition}\label{fin}
We label the three different faces of the cube by $(v_i)_{i=1\dots 3}$.
\end{definition}

\begin{definition}\label{dircomdef}
Consider the billiard map $T$ inside the cube, and a point
$(m,\omega)\in X_\infty$. We define the complexity $p(n,m,\omega)$
by the complexity of the infinite word $v_{m,\omega}$ (see Definition \ref{compu}). We call it the \emph{directional complexity}.
\end{definition}

\subsection{Unfolding: definition and example}
The \emph{unfolding} is a very useful tool in the study of billiards behavior. 
Consider a billiard trajectory in a polyhedron. To
draw the orbit, we must reflect the line each time it hits a face
of the polyhedron. The unfolding consists in reflecting the
polyhedron through the face and continuing on the same line.

Althought we deal with the cube the figures are made in the case of the square.

\begin{example} 
Example of the cube.\\
The billiard orbit of $(m,\omega)$ appears to be as the
sequence of intersections of the line $m+\mathbb{R}\omega$ with
the lattice $\mathbb{Z}^3$, see Figure \ref{carrefig}. In the left picture we represent one billiard orbit inside the square on dash points. It is unfolded in the line which intersects $\mathbb{Z}^2$. 

\begin{figure}[h]
\includegraphics[width= 4cm]{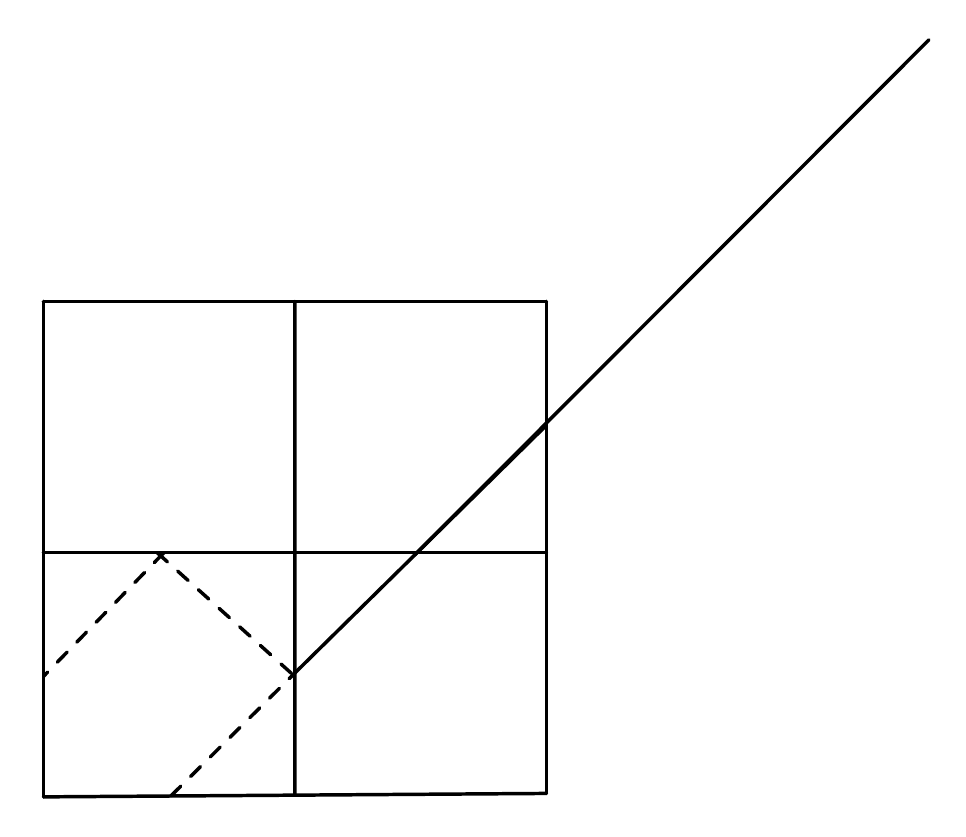}
\includegraphics[width= 4cm]{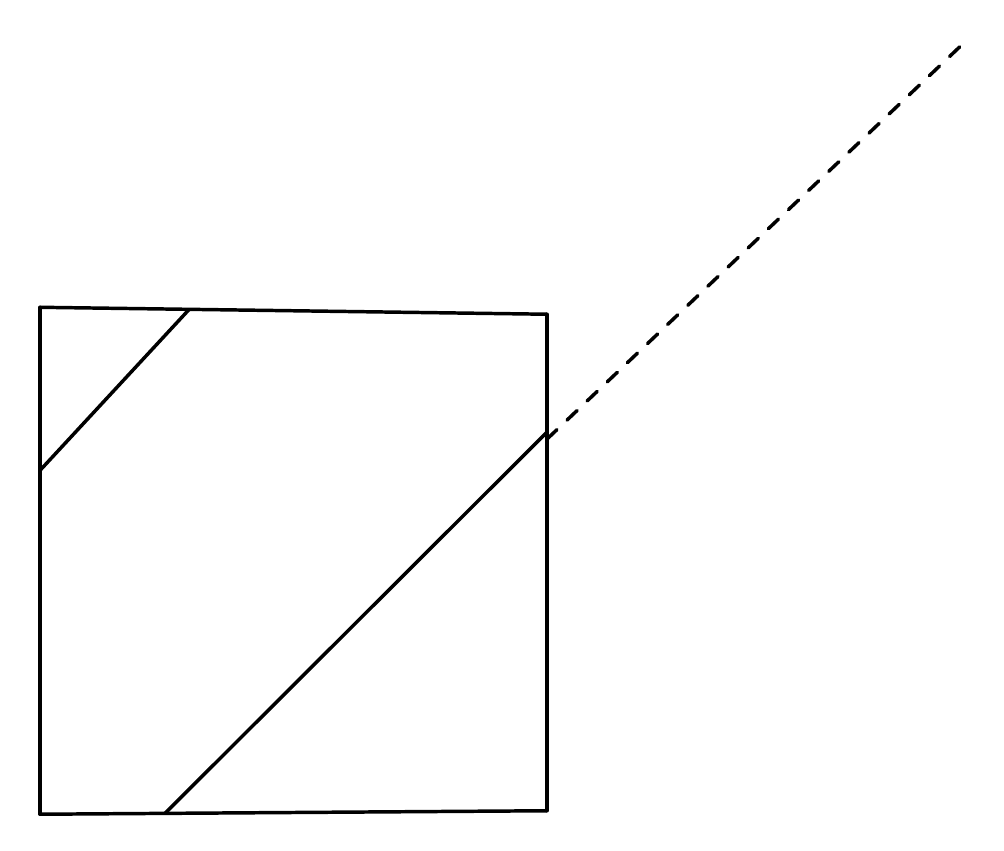}
\caption{Unfolding.}\label{carrefig}
\end{figure}

On the right picture we see that the study of the billiard orbit can be made on the big square where we identify the opposite sides. Then we obtain a torus, and the map is a translation on this torus:
\end{example}

 \begin{definition}\label{translation}
 For $\omega\in\mathbb{R}^3$, a translation $T_\omega$ of the torus is a map defined as follows.
 $$\mathbb{R}^3/\mathbb{Z}^3\rightarrow \mathbb{R}^3/\mathbb{Z}^3$$
$$T_\omega :(x,y,z)\mapsto (x,y,z)+\omega.$$
\end{definition}

Figure \ref{carrefig} explains the following result:
\begin{lemma}\label{unfold}
Let $\omega\in\mathbb{R}^3$, and consider the billiard map $T$ in a cube. \\
Then it is equivalent to study the orbit $(T^n(m,\omega))_n$ or the orbit $(T_\omega^n(m))_n$. 
 \end{lemma}

\subsection{Minimality}

\begin{definition}
A direction $\omega\in\mathbb{RP}^2$ is called a minimal direction
if for all point $m$, the sequence $(T^n(m,\omega)\cap
\partial{C})_{n\in\mathbb{N}}$ is dense in $X_{\infty}$.
\end{definition}
The following lemma deals with minimality of billiard words in the one-dimensional case. This minimality 
depends on algebraic properties of the translation direction. It will be used in the proof of the last cases of our main theorem.

\begin{lemma}\label{carre}
Let $\omega=(a,b)$ be an unit vector of
$\mathbb{R}^2$.  Consider the
billiard map in the square. Then

$i)$ The direction $\omega$ is a minimal direction if and
only if $a,b$ are rationally independent over $\mathbb{Q}$.

$ii)$ If the direction is not a minimal one, then for all point $(m,\omega)\in X_\infty$,
the billiard orbit of $(m,\omega)$ is periodic.
\end{lemma}

\begin{lemma}
In the cube, a direction $\omega$ is a minimal direction if and
only if the numbers $(\omega_i)_{i\leq 3}$ are independent over
$\mathbb{Q}$.
\end{lemma}
The proof of this lemma is based on Kronecker's lemma, see \cite{Ha.Wr}. 

\subsection{Billiard complexity}
We consider the coding of the billiard map defined in Section 2.3. 

\begin{definition}
For any $n\geq1$ let $s(n):=p(n+1)-p(n)$. For $v \in \mathcal{L}(n)$ let

$$m_{l}(v)= card\{a\in \mathcal{A},va\in \mathcal{L}(n+1)\},$$
$$m_{r}(v)= card\{b\in \mathcal{A},bv\in \mathcal{L}(n+1)\},$$
$$m_{b}(v)= card\{a\in \mathcal{A}, b\in \mathcal{A},bva\in \mathcal{L}(n+2)\}.$$

A word is called \emph{right special} if $m_{r}(v)\geq 2$, \emph{left special} if
$m_{l}(v)\geq 2$ and \emph{bispecial} if it is right and left special.
Let $\mathcal{BL}(n)$ be the set of bispecial words of length $n$. 
\end{definition}

Cassaigne \cite{Ca} has proved the following result, which can be also found in \cite{Ca.Hu.Tr}:
\begin{lemma}
$$s(n+1)-s(n)=\sum_{v\in \mathcal{BL}(n)}{[m_{b}(v)-m_{r}(v)-m_{l}(v)+1]}.$$
\end{lemma}

\begin{lemma}\label{Ta}\cite{Ta}
For a minimal direction the directional complexity is independent of the
initial point $m$.
\end{lemma}

{\bf Notations}: This result implies that we can omit the
initial point in the notation $p(n,m,\omega)$, with the assumption
that the orbit of $(m,\omega)$ is dense in $X$. In other cases we will denote by $p(n,\omega)$ the maxima of $p(n,m,\omega)$ over all admissible points $m$, see Definitions \ref{defbill} and \ref{dircomdef}.

\begin{definition}
In a polyhedron a generalized diagonal of direction $\omega$
between two edges is the union of all the billiard trajectory of
direction $\omega$ between two points of these edges. We say it is
of length $n$ if each billiard trajectory hits $n$ faces between
the two points.
\end{definition}
If we fix the initial edge we can describe the edges of length
$n$ by the following result.
\begin{lemma}\cite{moi2}\label{long}
Fix an edge $A$ of the initial cube. The edge $B$ is at length $n$
from the edge $A$ if and only  if for all point
$(b_1,b_2,b_3)$ of $B$ we have
$$\lfloor b_1\rfloor+\lfloor b_2\rfloor+\lfloor b_3\rfloor=n.$$
\end{lemma}
 We recall the result of \cite{moi1}
which will be useful in the following.
\begin{proposition}\label{d}
Assume the cube is coded with three letters such that two parallel faces correspond to the same letter.
Let $\omega$ be  an unit vector, which  is minimal for the cubic
billiard, then for all integer $n$ we have
$$s(n+1,\omega)-s(n,\omega)=N(n,\omega),$$
where $N(n,\omega)$ is the number of generalized diagonals of
direction $\omega$ and length $n$.
\end{proposition}
With the same hypothesis the next lemma proves that we can
construct at most two diagonals of combinatorial length $n$ in
this direction.
\begin{lemma}\label{diag}
If  $\omega$ is minimal for the billiard map inside a cube, then
we have
$$N(n,\omega)\leq 2\quad \forall n\in\mathbb{N}^*.$$
\end{lemma}

\begin{figure}[hbt]
\begin{center}
\includegraphics[width= 6 cm]{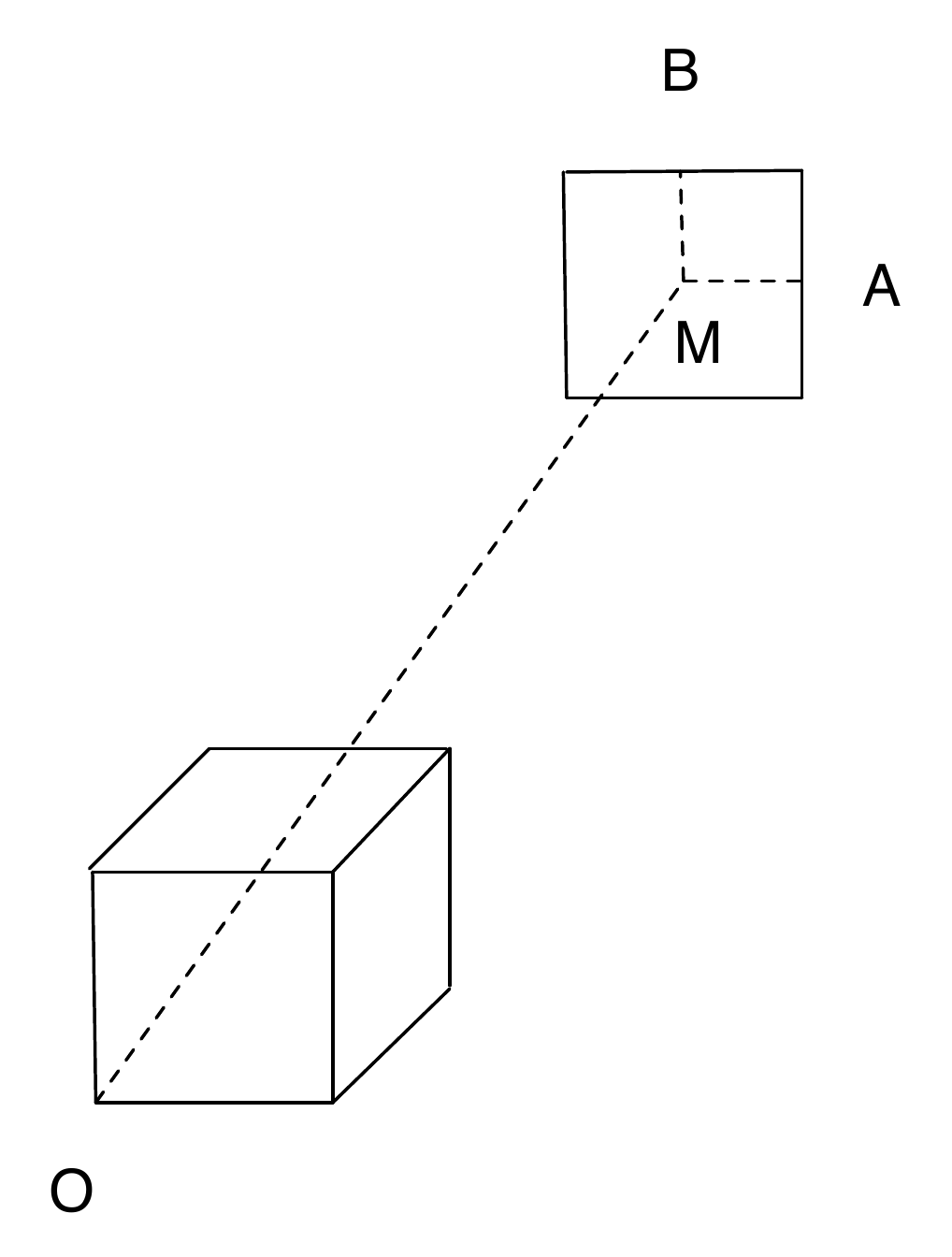}
\caption{Generalized diagonal in the cube.}\label{fig2}
\end{center}
\end{figure}
\begin{proof}
Let $O$ be a vertex of the cube and consider the segment of
direction $\omega$ which starts from $O$ and ends at a point $M$
after it passes through $n$ cubes. $M$ is a point of a face of an
unfolding cube, if we translate $M$ with a direction parallel to
one of the two directions of the face we obtain a point $A$ on an
edge and if we call $C$ the point such that
$\vec{OC}=\vec{MA}$\quad then $CA$ is a generalized diagonal, and
we have another one, $DB$ in the figure, arising from the second
translation.

The symmetries of the cube imply that these diagonals are the
only ones. It remains to prove that the two generalized diagonals
are of combinatorial length $n$.

The first thing to remark is that the condition of total
irrationality implies that a generalized diagonal can not begin
and end on two parallel edges.

To see that the combinatorial length is at most $n$ we can remark that the
sum of the length of the projections is twice the length of the trajectory,
so we just have to prove it for the projection, i.e.\ billiard in the
square, where it follows from the symmetry.
\end{proof}
The following Lemma recall some usual results from \cite{Pyt, Ta}.
It deals with minimality of billiard word in the one-dimensional case, depending on algebraic properties of the translation direction. 
\begin{lemma}\label{carre}
Consider a square coded with two letters.

$i)$ If $\theta$ is a minimal direction then we obtain
$p(n,m,\theta)=n+1$ for all $m$.


$ii)$ The orthogonal projection of a cubic billiard trajectory
on a face of the cube is a billiard trajectory inside a square.
\end{lemma}

To conclude this section, we recall a complexity result for a linear flow inside a polygon.
\begin{lemma}\cite{Hu}\label{Hub}
Consider a minimal linear flow on a polygon with parallel opposite sides.  Code the flow with a letter by sides, then the orbit of a point is coded by an infinite word of sub-linear complexity. Moreover, the complexity does not depend on the initial point and on the direction.
\end{lemma}

\section{Statement of the theorem}\label{enon}

\begin{theorem}
Fix an orthonormal basis of $\mathbb{R}^3$ such that the edges
of the cube are parallel to the coordinate axis.
Let $\omega=(\omega_1, \omega_2,
\omega_3)$ be a unit vector of $\mathbb{R}^3$ such
that $\omega_i\neq 0$ for all $i$. Denote
$\alpha=\frac{\omega_2}{\omega_1},\beta=\frac{\omega_3}{\omega_1}$.
Then assume one of the following holds
\begin{enumerate}
\item If $\alpha,\beta$ are rational numbers, then there
exists $C>0,n_0\in\mathbb{N}$ such that $p(n,\omega)=C$ for all integer $n\geq
n_0$.

\item If $\alpha$ is an irrational number, and $\beta$ is a
rational number, then there exists $C$ such that $p(n,\omega)\leq
Cn$.

\item If $\alpha,\beta$ are irrational numbers such that
$1,\alpha,\beta$ are linearly dependent over $\mathbb{Q}$, then
there exists $C$ such that $p(n,\omega)\leq Cn$ for all $n$.

\item If $\alpha,\beta, 1$ are linearly independent over
$\mathbb{Q}$, and if  $\alpha^{-1},\beta^{-1},1$ are linearly
dependent over $\mathbb{Q}$, then there exists $C\in]0;1[$ such
that $p(n,\omega)\sim Cn^2$.

\item  $p(n,\omega)=n^2+n+1$, in all other cases.
\end{enumerate}
\end{theorem}

\begin{remark}
The cases $(2)$ and $(3)$ correspond to the same algebraic
condition. We separate them, since in $(2)$ an
orthonormal projection on a face give a periodic word.

Except in the first case it is clear that the obtained billiard words are not ultimately periodic, thus we have $p(n,\omega)\geq n+1$.
\end{remark}
\begin{remark}
In the two last cases, the complexity is independent of the initial point $m$, by Lemma \ref{Ta}. In the first cases, we use the notations explained after Lemma \ref{Ta}.
\end{remark}

We study the dependance of the complexity function when the parameters vary.

\begin{corollary}\label{coroth1}
In case $(2)$, two directions with the same value of
$\beta$ have the same complexity.
\end{corollary}
This statement is a consequence of the last sentence of the proof of Case \ref{2}.
\begin{corollary}\label{coroth2}
For the third case, two directions in the same plane
have the same complexity. It means if two directions
$\omega,\theta$ satisfy
$a\omega_1+b\omega_2+c\omega_3=a\theta_1+b\theta_2+c\theta_3=0$
with $a,b,c\in\mathbb{Z}$, then $p(n,\omega)=p(n,\theta)$.\\
\end{corollary}
This point is a consequence of the last sentence of Case \ref{3}.
\begin{corollary}\label{coroth3}
 In case $(4)$, we can compute the constant $C$. If
$(\omega_i)$ satisfy the equation
$\frac{A}{\omega_1}=\frac{B}{\omega_2}+\frac{C}{\omega_3}$, with
$A,B,C\in\mathbb{N}$ then we obtain
$$C=1-\frac{1}{A(\alpha+\beta+1)}.$$
The other cases are obtained by permutation.
\end{corollary}
The last point is a consequence of  Lemma \ref{fin2} and Lemma \ref{fin1}, since $p=f_0$.

\section{Proof of the Theorem}\label{proof}
By Lemma \ref{unfold}, we will study the orbit of $(m,\omega)$ under the map $T_\omega$, see Definition \ref{translation}.
Each case of the theorem  will be treated separately; the first case reduces to a periodic one-dimensional coding, the second case to a billiard word inside a square, the third case reduces to a linear flow inside a polygon with parallel opposite sides, allowing to apply a result from Hubert, see Lemma  \ref{Hub}. The fourth case studies the complexity function by studying bispecial factors.

\subsection{First case} 
We prove the following result
\begin{proposition}
Assume the direction $\omega$ is such that
the numbers $\alpha,\beta$ are rational numbers. 
Then there exists $C>0,n_0$ such that $p(n,\omega)=C$ for all integer $n\geq
n_0$.
\end{proposition}
\begin{proof}
We study the orbit of the point $m_0=(x,y,z)$
under $T_\omega$. By unfolding we must compute the intersections
of the line $m_0+\mathbb{R}\omega$ with the three sort of faces.
The computation is similar in any case, thus we treat only the case of
the intersection with the face $Y=k$ (same thing for the faces
$X=l$ or $Z=m$ with $m,l,k\in\mathbb{Z}$). Suppose that there exists $\lambda$
such
that $\begin{pmatrix}x+\lambda,& y+\lambda \alpha,&
z+\lambda\beta\end{pmatrix}$ belongs to the
face $Y=k$. We obtain $\lambda=\frac{k-y}{\alpha}$, we deduce that the intersection
point is $(x+\frac{k-y}{\alpha},
k,z+\frac{k-y}{\alpha}\beta)$. The point of the cube
which corresponds in the unfolding to this point is
$(x+\frac{k-y}{\alpha}\quad
\mod\;1,0,z+\frac{k-y}{\alpha}\beta\;\mod\;1)$.

To obtain the sequence coding the orbit of $(x,y,z)$ by $T_\omega$, it
remains to make $k$ vary in $\mathbb{Z}$. Since $\alpha,\beta$ are two
rational numbers, we deduce that the sequence is periodic. Thus
the trajectory is periodic, and the complexity is an eventually
constant function.
\end{proof}
\subsection{Case number $2$}\label{2}
\begin{proposition}
Assume $\alpha$ is an irrational number, and $\beta$ is a rational number. Then there exists $C$ such that $p(n,\omega)\leq
Cn$, for all integer $n$.
\end{proposition}
\begin{proof}
 Consider the projection on the plane
$Oxz$. Since $\beta$ is a rational number, we have a periodic
trajectory in the square (see Lemma \ref{carre}). Denote by $(a_i)$
the periodic sequence of points inside the square, such that
$a_p=a_1$, denote $b_i$ the points of the cube such that
$(a_ib_i)$ is parallel to the axis $Oz$. Consider the union $S$ of
the intervals
$$[a_i,a_{i+1}], [b_i,b_{i+1}], [a_i,b_i] \quad
i\leq p-1.$$ The trajectory of $(m_0,\omega)$ is included in $S$,
as can be seen by projection, see the left part of Figure \ref{f1}. Now unfold
the trajectory. The unfolding of $S$ is a rectangle. This
rectangle is partitioned in several rectangles of the same shape.
The trajectory is a translation in this rectangle, see right part of Figure
\ref{f1}. This translation is coded with three letters and it is minimal by hypothesis on $\alpha$. If the
translation was coded by two letters we would obtain a sturmian
word. The computation of the complexity is reduced to the
computation of the complexity of a translation: it is clearly sub-linear. Moreover, remark that the rectangle $S$ only depends on
$\beta$ by construction.
\end{proof}


\begin{figure}
\includegraphics[width= 4cm]{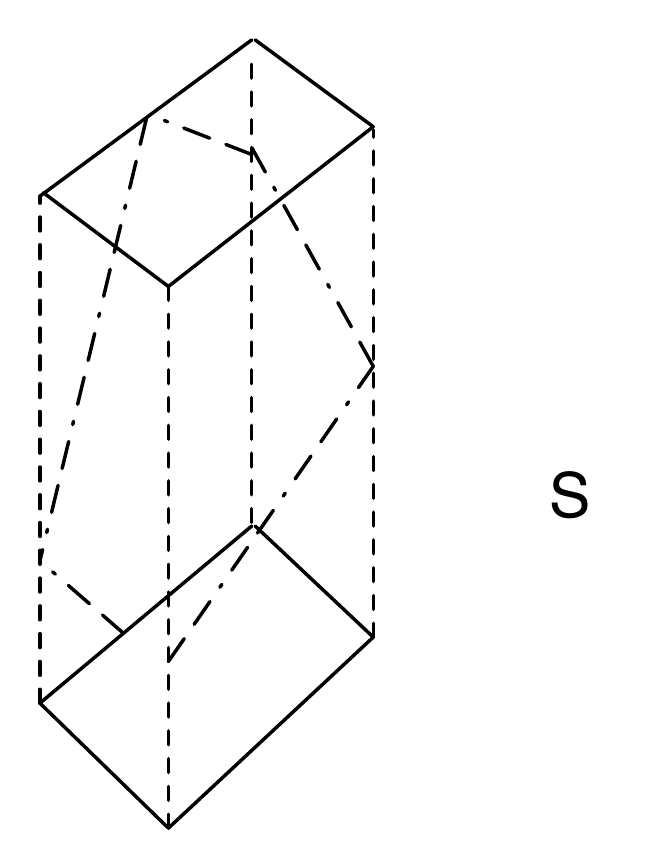}
\includegraphics[width= 4cm]{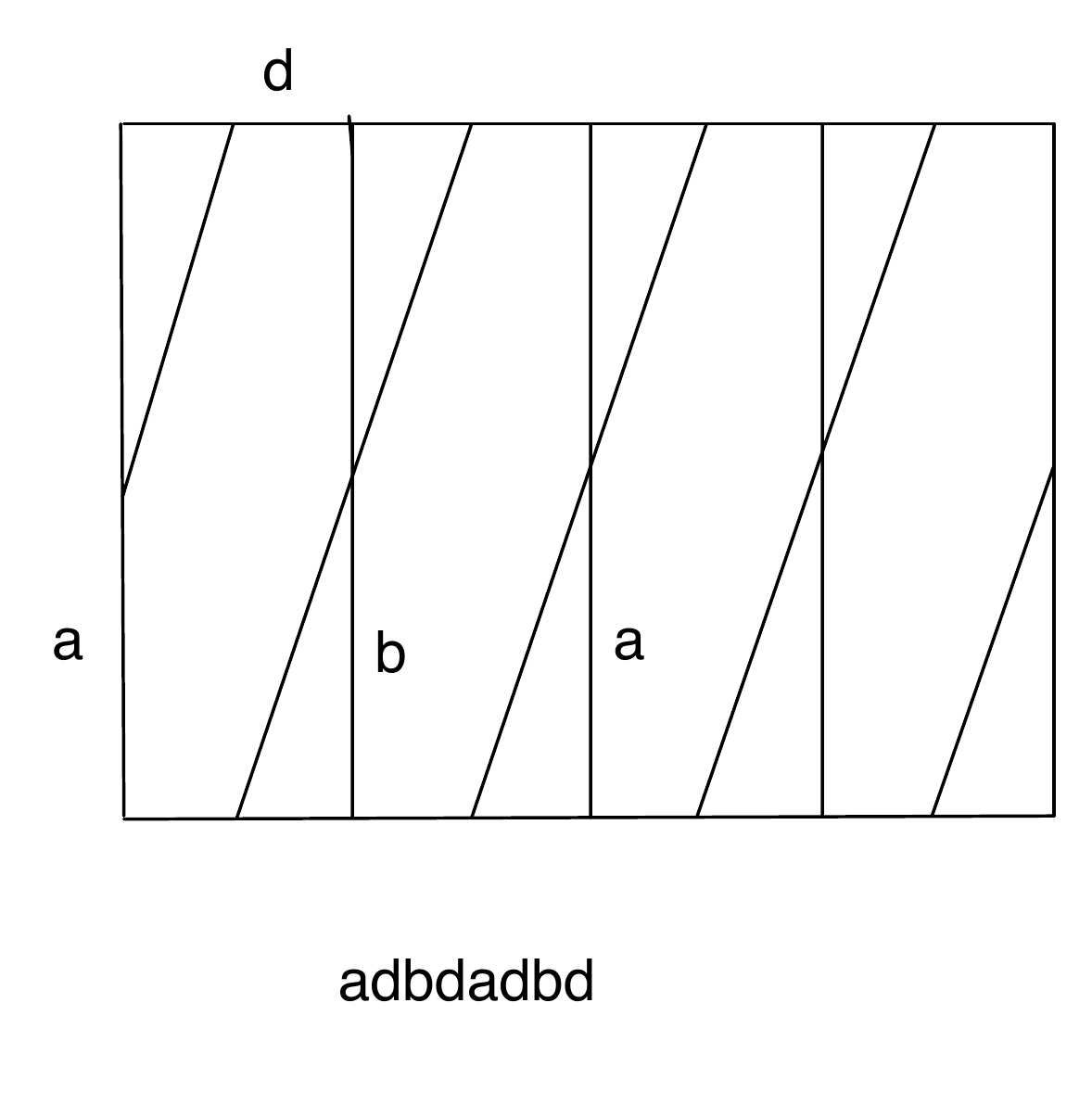}
\caption{Billiard trajectory inside $S$, and unfolding.}\label{f1}
\end{figure}

\subsection{Case number 3}\label{3}
Let $(x,y,z)\in\mathbb{R}^3$, and $a,b,c\in\mathbb{N}$. Consider the plane $P$ of equation $cX+aY+bZ=cx+ay+bz$.
Then consider the canonical projection 
$$\pi: \mathbb{R}^3\mapsto \mathbb{R}^3/\mathbb{Z}^3.$$
The plane intersects the cubes of $\mathbb{Z}^3$ into polygons,
we use this projection to translate the polygons inside the initial cube. Indeed, the initial cube can be identified with  $\mathbb{R}^3/\mathbb{Z}^3.$

\begin{lemma}\label{plans}
The set $\pi(P)$ is the union of a finite number of polygons.
\end{lemma}
\begin{proof}
Consider the intersection $P$ of the plane with the initial
cube. The other polygons are obtained by translating the
intersection  of $P$ with another cube. Thus the study of the
edges of the polygons in the face $Z=0$ can be made by looking at
the edges in the face $Z=k$, when $k$ takes values in $\mathbb{Z}$.
Consider the intersection of $P$ with the face $Z=k$. We obtain a
line of equation
$$\begin{cases}Z=k\\cX+aY=cx+ay+bz-bk\end{cases}$$ The slope of this line is $\frac{-c}{a}$.
It is a rational number.
When $k$ changes this slope is constant, thus all the edges in
this face are parallel. Moreover the intersections of this line
with the edges of the cube are obtained by replacing $Y$ or $X$ by
an integer $l$. For example we obtain
$$X_{k,l}=\frac{cx+a+bz-bk-al}{c}=\frac{cx+ay+bz}{c}-\frac{bk+al}{c}\quad \mod1.$$
The set of all points is obtained by taking the union of $k,l$ in
$\mathbb{Z}$. This gives a finite number of points since these
numbers are rational. Thus in each face there are a finite number
of parallel edges. Moreover inside two parallel faces the edges
are parallel.
\end{proof}

\begin{proposition}
Assume that:
 $\alpha,\beta$ are irrational numbers such that
$1,\alpha,\beta$ are linearly dependent over $\mathbb{Q}$. Then, 
there exists $C>0$ such that $p(n,\omega)\leq Cn$.
\end{proposition}
\begin{proof}
 Consider the relation
$a\alpha+b\beta+c=0$ with $a,b,c\in\mathbb{Z}$. We will study the
orbit of $m_0=(x,y,z)$ under $T_\omega$. A point on this line has
coordinates
$(\lambda+x,
\lambda\alpha+y, \lambda\beta+z)$. Thus it is in the
plane $cX+aY+bZ=cx+ay+bz$. This plane intersects each unity cube
of the lattice $\mathbb{Z}^3$ in a polygon. By a translation each
polygon is shifted to the initial cube. This union of polygons
contains the orbit of a point ,see Figure \ref{fig3}.
By Lemma \ref{plans}, there is a finite number of polygons.
Now, the orbit of a point is included inside this finite union of
polygons. The opposite sides of these polygons are parallel. Thus
the billiard flow becomes a linear flow inside a polygon with
parallel opposites sides. We apply the result of Lemma  \ref{Hub}.
Here we remark that several edges can be coded by the same letter,
thus the complexity can be less than the initial one. To end the proof, we remark that the complexity only depends on the polygon. Hence it only depends on $a,b,c$.
\end{proof}

\begin{figure}
\includegraphics[width= 4cm]{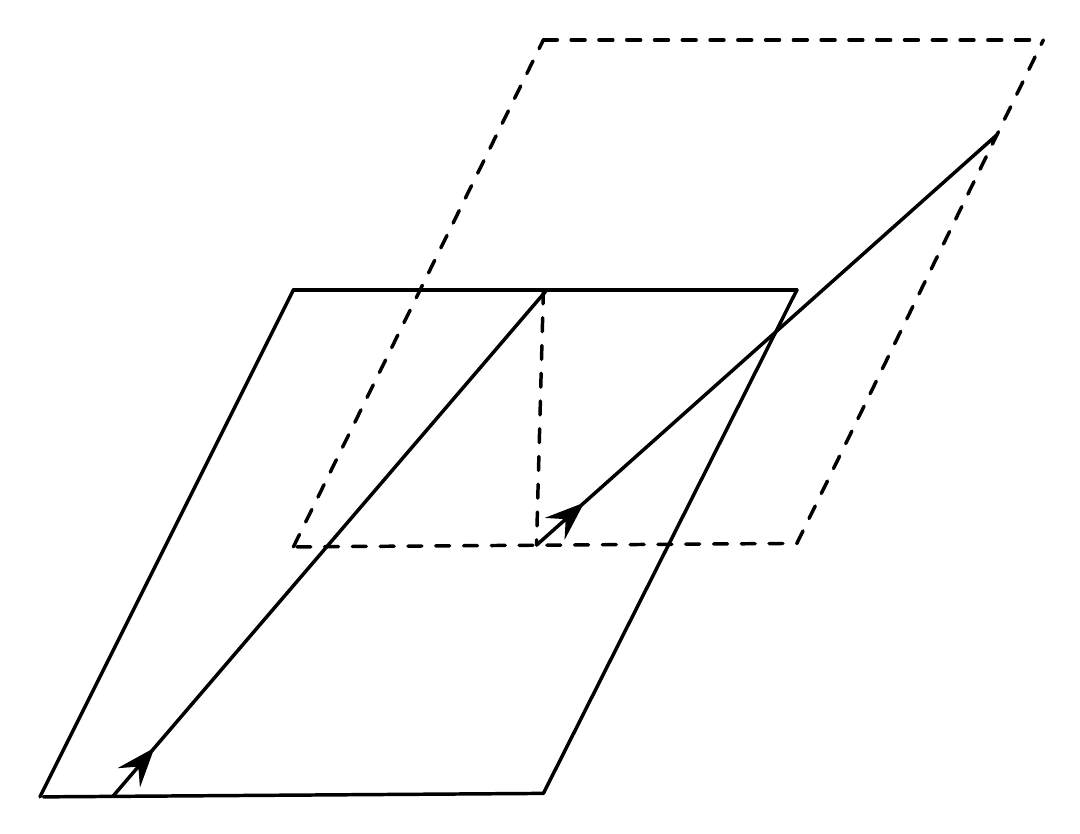}
\caption{Billiard map inside the union of polygons}\label{fig3}
\end{figure}

 \subsection{Case number 4}
 
In this section we will show that the number of generalized
diagonals in the direction $\omega$ can be strictly less than two.
First of all we recall the following lemma.
\begin{lemma}\label{syst}
Consider three numbers $a,b,c$ linearly independent over
$\mathbb{Q}$. Assume that the following equation
$$x/a+y/b+z/c=0,$$ has an integer solution $(x,y,z)$ with $x\neq
0$. Then the rational solutions of the equation are :
$$r(x',\frac{yx'}{x},\frac{zx'}{x})\quad x',r\in\mathbb{Q}.$$
\end{lemma}

\begin{proof}
Consider two rational solutions:
$$\begin{cases}x/a+y/b+z/c=0\\ x'/a+y'/b+z'/c=0\\ \end{cases}
\begin{cases}x/a+y/b+z/c=0\\ (yx'-xy')/b+(zx'-xz')/c=0\\ \end{cases}$$
Since $b/c$ is an irrational number, we deduce
 $$\begin{cases}x/a+y/b+z/c=0\\  yx'=xy'\\zx'=xz'\\ \end{cases}$$

$$\begin{cases}y'=yx'/x\\ z'=zx'/x\\ \end{cases}$$
\end{proof}

\begin{lemma}\label{ordre}
Assume there exists $n$ such that $N(n,\omega)<2$ then :

$$s(n+1,\omega)-s(n,\omega)=0.$$
Moreover there exists a line of direction $\omega$ which
intersects the three types of edges (see Definition \ref{numeroedge}) and these three edges are in a
fixed order, given by the direction.
\end{lemma}
\begin{proof}
First recall that the minimality of $\omega$ implies that the
edges of a diagonal in direction $\omega$ can not be parallel. In
the rest of the proof we can assume that the edges of the
generalized diagonal of direction $\omega$ are of type 1 and 3, see Definition \ref{numeroedge}.
\begin{enumerate}
\item Consider a trajectory in direction $\omega$ between two
edges of types 1 and 3, consider the orthogonal reflection over
the plane $X=Z$. This map exchanges the edges of type $1$ and $3$,
but it leaves invariant edges of type $2$. That implies that
$N(n,\omega)$ is an even number, thus we can not have
$N(n,\omega)=1$. Hence we have $N(n,\omega)=0$. Proposition \ref{d}
finishes the first part of the proof.

\item By applying a translation we can always assume that the
intersection points of the line $m+\mathbb{R}\omega$ with the
edges of the cube have for coordinates
$$(x,0,0);(a,y,b);(c,d,z),$$
with $x,y,z$ reals numbers and $a,b,c,d$ integers. 

We obtain the system

$$\begin{cases}
x+\lambda\omega_1=a\\
\lambda\omega_2=y\\
\lambda\omega_3=b\\
x+\mu\omega_1=c\\
\mu\omega_2=d\\
\mu\omega_3=z\\
\end{cases}$$
where $\lambda,\mu$ are real numbers.

$$\begin{cases}
\lambda\omega_2=y\\
\lambda\omega_3=b\\
\mu\omega_2=d\\
\mu\omega_3=z\\
x=a-b\frac{\omega_1}{\omega_3}\\
\frac{a-c}{\omega_1}=\frac{b}{\omega_3}-\frac{d}{\omega_2}\\
\end{cases}$$
By hypothesis on $\omega$, we have a relation of the form
$$\frac{A}{\omega_1}+\frac{B}{\omega_2}+\frac{C}{\omega_3}=0,$$
where $A,B,C\in\mathbb{Z}$ are relatively prime.

 The last equation of the system is of the same form:
$$\frac{a-c}{\omega_1}+\frac{-b}{\omega_3}+\frac{d}{\omega_2}=0.$$
We apply Lemma \ref{syst} and one deduces that $A$ is a
divisor of $a-$c and:
$$\begin{cases}
d=B\frac{a-c}{A}\\
-b=C\frac{a-c}{A}
\end{cases} $$
Finally the system becomes

$$\begin{cases}
\lambda\omega_2=y\\
\lambda\omega_3=b\\
\mu\omega_2=d\\
\mu\omega_3=z\\
x=a-b\frac{\omega_1}{\omega_3}\\
d=B\frac{a-c}{A}\\
-b=C\frac{a-c}{A}
\end{cases}$$
This system has at least one solution. Hence the existence of
the line is proved.\\

 This system allows us to make several remarks. First, the coordinates
 $\omega_i$ are positive numbers.
This implies that $A,B,C$ can not all be positive numbers. Assume
that we have $A<0, B>0, C>0$ (the other cases are similar). We
deduce that $a-c$ and $d$ are of opposite sign, and that $a-c$ and
$b$ are of same sign.

\item Assume $a-c\geq 0$ this implies $$ d\leq 0, b\geq
0.$$ From the system we deduce $$\lambda\geq 0, \mu\leq 0.$$ This
implies that the edges appear in the order $3;1;2$.

\item On the other hand, if $a-c\leq 0$, by a similar
argument, we have that the order is $2;1;3$.

Moreover the two orders are correlated, it depends on the direction that it used to move along the line. Hence we can reduce to one order.
\end{enumerate}
\end{proof}

\begin{corollary}\label{cor}
Assume $\omega$ is a minimal direction and satisying
$$\frac{A}{\omega_1}=\frac{B}{\omega_2}+\frac{C}{\omega_3}\quad A,B,C\in\mathbb{N}^*.$$
Then for all integer $n$ we have the dichotomy: Either the billiard
orbit of the origin, at step $n$, meets a face labelled by
$v_1$ (see Definition \ref{fin}), and if $A$ divides $n$, then $s(n+1,\omega)-s(n,\omega)=0$,
otherwise $s(n+1)-s(n)=2.$
\end{corollary}
\begin{proof}
First we claim that there exist an infinite number of integers
$n$ such that $s(n)=0$. Indeed in the last system obtained in the
proof of Lemma \ref{ordre} we can modify the values of $a,c$ such
that $A$ divides $a-c$ . Now we can assume that the order related
to the edges is $3;1;2$ see Lemma \ref{ordre}. Consider the orbit
of the origin, and the intersection with a face (of a cube of
$\mathbb{Z}^3$) parallel to $X=0$. With the method of Lemma
\ref{diag} we deduce that the only possibility for a generalized diagonal is a
trajectory between edges $3$ and $2$. Denote by $n$ the length of
the diagonal, by previous system we deduce that if $A$ divides
$n$ the trajectory between $3$ and $2$ passes through the edge $1$.
We deduce $s(n+1,\omega)=s(n,\omega)$. The first part is proved.

Assume now that we meet another face at step $n$, for example the
face parallel to $Z=0$. Then the two associated diagonals have for
order $1;2$ and $2;1$. We prove by contradiction that we can not
have $N(n,\omega)\leq 1$. Since the order is unique, see Lemma
\ref{ordre}, the only possibility to obtain a third edge is to
start from the edge labelled 1. Then the diagonal which starts form
$3$ does not intersect another edge. This implies $N(n,\omega)=1$,
but this is a contradiction with the first part of Lemma
\ref{ordre}.
\end{proof}
This corollary implies that the sequence
$(s(n,\omega))_{n\in\mathbb{N}}$ can take only two values. Due to
the next lemma, to finish the proof it remains to obtain the
frequency of each value.
\begin{lemma}\label{fin1}
Assume that the sequence
$(s(n+1,\omega)-s(n,\omega))_{n\in\mathbb{N}}$ has value in
$\{0;1;2\}$, and that the numbers $0;1;2$ have respectively for
frequency $l,m,p$. Then the complexity satisfies
$$p(n)\sim\frac{m+2p}{2}n^2.$$
\end{lemma}
\begin{lemma}\label{fin2}
Assume the direction satisfy the hypothesis
$\frac{A}{\omega_1}=\frac{B}{\omega_2}+\frac{C}{\omega_3}$, with
$A,B,C\in\mathbb{N}$. Then the frequency $l$ of $0$ in the sequence
$(s(n+1,\omega)-s(n,\omega))_{n\in\mathbb{N}}$ is:
$$l=\frac{\omega_1}{A(\omega_1+\omega_2+\omega_3)}.$$
\end{lemma}
\begin{proof}
By Corollary \ref{cor}, it is equivalent to consider the
intersection of the orbit of the origin with the planes parallel
to $X=0$. A point in the orbit of the origin has for coordinates:
$$\begin{pmatrix}\lambda\omega_1,& \lambda\omega_2,& \lambda\omega_3\end{pmatrix}.$$
It meets the face $X=iA$ at the point
$$\begin{pmatrix}iA,& \frac{iA}{\omega_1}\omega_2,& \frac{iA}{\omega_1}\omega_3\end{pmatrix}.$$
Then we must compute the number of $i$ such that this point is at
combinatorial length less than $n$. By Lemma \ref{long}, it remains
to compute
$$card\{i|iA+[\frac{iA}{\omega_1}\omega_2]+[\frac{iA}{\omega_1}\omega_3]\leq n\}.$$
$$=\frac{n\omega_1}{A(\omega_1+\omega_2+\omega_3)}+o(n).$$
We deduce the value of the frequency.
$$l=\frac{\omega_1}{A(\omega_1+\omega_2+\omega_3)}.$$
\end{proof}
 \begin{proposition}
Assume the numbers $\alpha,\beta, 1$ are linearly independent over
$\mathbb{Q}$, and $\alpha^{-1},\beta^{-1},1$ are linearly
dependant over $\mathbb{Q}$. Then
 there exists $C\in]0;1[$ such
that $p(n,\omega)\sim Cn^2$.
\end{proposition}
\begin{proof}
The proof is a consequence of Corollary \ref{cor}, Lemmas \ref{fin1}, \ref{fin2}.
\end{proof}
\subsection{Last case} The proof can be found in \cite{moi1} or in
\cite{Ba} for the $s$-dimensional case. In the first article the main object of the proof consists in stating that $N(n,\omega)=2$ for every integer $n$. In the second article, the main step of the proof consists in stating that $p(n,\omega)$ does not depend on the direction $\omega$.

\medskip

{\bf Acknowledgements: } The author thanks the referees for the fruitful comments and remarks.
\bibliography{bibliorot}
\end{document}